\documentclass[12pt,reqno]{amsart}
\usepackage{amsfonts}
\usepackage{amssymb}
\usepackage{textcomp}
\usepackage{graphicx}
\def\sqr#1#2{{\vcenter{\vbox{\hrule height.#2pt
        \hbox{\vrule width.#2pt height#1pt \kern#2pt
        \vrule width.#2pt}
        \hrule height.#2pt}}}}

\newcommand{\nc}{\newcommand}
\nc{\parent}[1]{$[\![#1]\!]$}
\newtheorem{theorem}{Theorem}[section]

\newtheorem{corollary}{Corollary}[section]
\newtheorem{proposition}{Proposition}[section]
\newtheorem{remark}{Remark}
\newtheorem{definition}{Definition}[section]
\newtheorem{assumption}{Assumption}[section]

\newenvironment{pf-main}{{\sc Proof of Theorem \ref{mainresult}.}\hspace{3mm}}{\qed}

\nc{\cadlag}{c\`{a}dl\`{a}g } \nc{\ba}{\begin{array}}
\nc{\ea}{\end{array}} \nc{\be}{\begin{equation}}
\nc{\ee}{\end{equation}} \nc{\bea}{\begin{eqnarray}}
\nc{\eea}{\end{eqnarray}} \nc{\bean}{\begin{eqnarray*}}
\nc{\eean}{\end{eqnarray*}} \nc{\bu}{\bullet} \nc{\nn}{\nonumber}
\nc{\cA}{{\mathcal A}} \nc{\cB}{{\mathcal B}} \nc{\cC}{{\mathcal
C}} \nc{\cD}{{\mathcal D}} \nc{\bbD}{\mathbb{D}}
\nc{\cG}{{\mathcal G}} \nc{\cF}{{\mathcal F}} \nc{\cS}{{\mathcal
S}} \nc{\cU}{{\mathcal U}} \nc{\cH}{{\mathcal H}}
\nc{\cK}{{\mathcal K}}\nc{\cL}{{\mathcal L}}  \nc{\cM}{{\mathcal
M}} \nc{\cO}{{\mathcal O}} \nc{\cP}{{\mathcal P}}
\nc{\bbE}{\mathbb{E}} \nc{\bbF}{\mathbb{F}}
\nc{\bbEQ}{\mathbb{E}_{\mathbb{Q}}} \nc{\eps}{\varepsilon}
\nc{\bbEP}{\mathbb{E}_{\mathbb{P}}}\nc{\bbL}{\mathbb{L}}
\nc{\what}{\widehat} \nc{\bbP}{\mathbb{P}} \nc{\bbQ}{\mathbb{Q}}
\nc{\del}{\partial} \nc{\Om}{\Omega} \nc{\om}{\omega}
\nc{\bbR}{\mathbb{R}} \nc{\bbN}{\mathbb{N}} \nc{\fps}{$(\Om, \cF,
(\cF_t)_{t\geq 0}, \bbP)$} \nc{\bbC}{\mathbb{C}}
\nc{\bfr}{\begin{flushright}} \nc{\efr}{\end{flushright}}
\nc{\dXt}{\Delta X_{t}} \nc{\dXs}{\Delta X_{s}}
\nc{\bs}{\blacksquare} \nc{\dX}{\Delta X} \nc{\dY}{\Delta Y}
\nc{\dnkx}{\left(X(T^{n}_{k})-X(T^{n}_{k-1})\right)}
\nc{\esssup}{\mathrm{ess}\mbox{ }\mathrm{sup}}
\nc{\essinf}{\mathrm{ess}\mbox{ } \mathrm{inf}}
\nc{\dhats}{\widehat{\delta_s}} \nc{\half} {\frac{1}{2}}

\nc{\ol}{\overline}
\def\rar{\rightarrow}

\nc{\chf}{\mbox{$\mathbf1$}}
\begin{document}

\title{Path transformations for local times of one-dimensional diffusions}
\author{Umut \c{C}etin}
\address{Department of Statistics, London School of Economics and Political Science, 10 Houghton st, London, WC2A 2AE, UK}
\email{u.cetin@lse.ac.uk}
\date{\today}
\begin{abstract}
Let $X$ be a regular one-dimensional transient diffusion and $L^y$ be its local time at $y$. The stochastic differential equation (SDE) whose solution corresponds to the process $X$ conditioned on $[L^y_{\infty}=a]$ for a given $a\geq 0$ is constructed and a new path decomposition result for transient diffusions is given. In the course of the construction {\em Bessel-type motions} as well as their SDE representations are studied.   Moreover, the Engelbert-Schmidt theory for the weak solutions of one dimensional SDEs is extended to the case when the initial condition is an entrance boundary for the diffusion. This extension was necessary for the construction of the Bessel-type motion which played an essential part in the SDE representation of $X$ conditioned on $[L^y_{\infty}=a]$.
\end{abstract}
\maketitle

\section{Introduction} 
Conditioning  a given Markov process $X$ is a well-studied subject which has become synonymous with the term {\em $h$-transform}. If one wants to condition the paths of $X$ to stay in a certain set, the classical recipe consists of finding an appropriate excessive function $h$,  defining the transition probabilities of the conditioned process via  $h$, and constructing on the canonical space a Markov process $X^h$ with these new transition probabilities using standard techniques. This procedure is called an $h$-transform and its origins go back to  Doob and his study of boundary limits of Brownian motion \cite{Doob1, Doob2}. If $h$ is a minimal excessive function with a pole at $y$ (see Section 11.4 of \cite{ChungWalsh} for definitions), then $X^h$ is the process $X$ conditioned to converge to $y$ and killed at its last exit from $y$. We  refer the reader to Chapter 11 of \cite{ChungWalsh} for an in-depth analysis of $h$-transforms and their connections with time reversal and last passage times. 

If $X$ is a regular transient diffusion taking values in some subset of $\bbR$, $h:=u(\cdot,y)$ is a minimal excessive function for every $y$ in its state space, where $u$ is the potential density of $X$. Moreover, $y$ is the unique pole of this excessive function. Thus, the preceding discussion suggests that this $h$-transform conditions $X$  to converge to $y$ and kills it at its last exit from $y$. For a thorough discussion of $h$-transforms for one-dimensional diffusions and the proofs of certain results that are considered to be folklore in the literature we refer the reader to a recent manuscript by Evans and Hening \cite{EH}. The recent works of Perkowski and Ruf \cite{PerRuf} and Hening \cite{Hening} also consider  specific cases of conditioning for one-dimensional diffusions. 

In this paper we are interested in conditioning a one-dimensional regular transient diffusion on the value of its local time  at its lifetime. We assume that the diffusion cannot be killed in  the interior of its state space. It is well-known that $(X,L^y)$ is a two-dimensional Markov process, where $L^y$ is the local time of $X$ at level $y$. If we would like to apply an $h$-transform to achieve our conditioning, we need to find a minimal excessive function of the pair $(X,L^y)$ with a suitable pole so that the local time of the $X^h$ equals a given number, say, $a\geq 0$ at its lifetime. The problem with this approach is that it requires the knowledge of the potential density of the Markov pair $(X,L^y)$, which is in general not easily obtained or characterised. Moreover, it will require a killing procedure.

We shall follow a different approach and construct the conditioned process as a weak solution to a stochastic differential equation (SDE) with a suitably chosen drift and an additional  randomisation, which in essence is what deviates our approach from that of an $h$-transform. Moreover, there will not be any killing involved. On our way to constructing this SDE we will obtain the following contributions that are of significant interest in their own right:
\begin{itemize}
\item {\em An extension of the Engelbert-Schmidt theory for the weak solutions of one-dimensional SDEs.} The Engelbert-Schmidt theory constructs the weak solutions of one-dimensional SDEs as a time and scale change of a Browanian motion. This construction fails if the initial condition is an entrance boundary. We extend the theory when the initial condition is an entrance boundary by a time and scale transformation of a 3-dimensional Bessel process. 
\item {\em SDE representation for Bessel-type motions.} We obtain the SDE representation for the excursions of $X$ away from a point conditioned to last forever.   
\item {\em A new path-decomposition result for transient diffusions.} We will show that we can obtain the transient diffusion by suitably pasting together a recurrent transformation, which is introduced in \cite{Crt}, and a Bessel-type motion. As such, this will give an alternative way of simulating one-dimensional diffusions.
\end{itemize}

Returning back to  our main point of interest, that is, constructing an SDE whose solution coincides in law with $X$ conditioned on $L^y_{\infty}$, we shall next see in brief how the above contributions play a role in this construction.  Since we are interested in obtaining an SDE for the conditioned process this obviously necessitates the original process, $X$,  being a solution of  an SDE. In Section \ref{s:recipe} we impose the standard Engelbert-Schmidt conditions in order to ensure that $X$ itself is the unique weak solution of an SDE upto a, possibly finite, exit time from its state space. Our aim is to  construct an SDE -- for which  weak uniqueness holds --  such that the law of its solutions coincides with the law of $X$ conditioned on the null set $[L^y_{\infty}=a]$. As such, there will be no killing involved in the interior of the state space  in our conditioning. 

At the end of Section \ref{s:recipe} we give a recipe for constructing the SDE with the above property and a brief warning about pitfalls that one may encounter depending on the limiting values of $X$. However, difficulties aside, the desired conditioning should be done in two steps: In Step 1 we have to choose a suitable drift that makes sure that the solution keeps hitting $y$ until the local time process equals $a$. As soon as this is achieved, in Step 2, we have to choose a new drift that prevents the solutions hitting $y$ again and, thus, keeping the local time process constant at $a$. 

Since we want to make sure that $L^y_{\infty}=a$ with probability one for the conditioned process, it is necessary that we have to choose a drift term in Step 1 to transform $X$ into a recurrent process. Indeed, if the solution of the SDE considered in Step 1 is  transient, there is a non-zero probability for its solution to drift away to the cemetery state before its local time process hits $a$. To this end we use the concept of a {\em recurrent transformation} that is developed in \cite{Crt}. A particular example of a recurrent transformation that is borrowed from \cite{Crt} will give us the drift that achieves the Step 1 of our conditioning procedure. 

Section \ref{s:excursion} prepares the drift terms that one would use in Step 2 of the conditioning procedure. Since the solution of the SDE  at the end of Step 1 equals $y$, and the new drift should be chosen in a way to keep the conditioned process away from $y$, the drift term that must be employed in the second step of our construction leads to  solutions that are linked to the excursions of $X$ away from $y$. Indeed, if $y=0$ and  $X$ is a Brownian motion killed at $0$, the  SDE that we obtain in Step 2 is the SDE associated to the 3-dimensional Bessel process. In Section \ref{s:excursion} we give the SDE characterisations of these {\em Bessel-type motions} -- a term adapted from McKean  in his study of excursions of diffusions \cite{MKexc} -- and prove a time reversal result akin to those that can be found in the seminal paper of Williams \cite{W73} using a theorem due to Nagasawa \cite{Nagasawa}. 

After all these preparations and the other related results of independent interest, the SDE that is associated to the desired conditioning is constructed in Section \ref{s:main}. Weak uniqueness of its solutions is proven in Theorem \ref{t:Lbrdige} and Corollary \ref{c:corB} establishes that the law of its solutions possesses the desired bridge property using an enlargement of filtration technique. Corollary \ref{c:corB} also paves the way for a new path decomposition result for transient diffusions.  Theorem \ref{t:pathdecomp} is reminiscent of Williams' path decomposition for the Brownian motion killed at $0$ and constructs the original process, $X$, by pasting together a recurrent transformation and a time-reversed Bessel-type motion.

Another contribution of this paper is to the solutions of one dimensional time-homogeneous SDEs. In Step 2 we construct an SDE for a one-dimensional diffusion, where the initial condition is an entrance boundary\footnote{In the terminology of Ito-McKean this corresponds to an entrance-not-exit boundary.}. The existence and uniqueness of solutions to such SDEs do not follow from the Engelbert-Schmidt theory, which constructs the solution  from a time-change applied to a Brownian motion. As explained in Section \ref{s:recipe}, the reason for the non-applicability of the available theory is due to   the scale functions being infinite at entrance boundaries forcing one to start the Brownian motion at $+\infty$  if one wants to use the method of Engelbert and Schmidt.  Theorem \ref{t:SDEent} extends the Engelbert-Schmidt theory to the case when the initial condition is an entrance boundary. The proof uses  similar tools employed in the proof of the classical result of Engelbert and Schmidt without the entrance boundary. However, we construct the  weak solution  as a time-changed Bessel process as opposed to a time-changed Brownian motion.

Apart from the contributions listed above another motivation of this work stems from the studies of financial equilibrium under asymmetric information in the setting of Kyle \cite{Kyle} and Back \cite{Back}.   Consider a financial asset that is traded in a market that consists of an insider who knows the value of the asset, market makers and uninformed traders. Suppose that the asset is an annuity that pays $r(X_t)dt$ over the time interval $[t,t+dt)$, where $r$ is an increasing nonnegative function and $X$ is a state variable depending on the total demand of the asset. As in \cite{Krh} and \cite{Back-Baruch} let us also suppose that trading comes to an end at an exponential time with parameter $\alpha$. Then the equilibrium considerations as in \cite{Krh} imply that the optimal strategy of the insider is given by the drift of a particular transient  one-dimensional diffusion $Y$n  conditioned on the event $\{\exp(\int_0^{\infty}r(Y_t)dt)=V\}$, where $V$ is a random variable only known by the insider. 

Note that $\int_0^{\infty}r(Y_t)dt$ is  a {\em perfect continuous additive functional} (PCAF) of $Y$. Thus, understanding the behaviour of diffusions conditioned on the terminal value of a continuous additive functional will be quite useful for the aforementioned models of financial equilibrium. The present paper is the first step in this direction given that every PCAF $A$ of a one-dimensional diffusion can be written as
\[
A_t= \int_l^r L^y_t \mu_A(dy),
\]
where $L^y$ is the local time at level $y$ and $\mu_A$ is the so-called {\em Revuz measure} (see, .e.g., Exercise 75.12 in \cite{GTM}). Although the analogous study of conditioning with general PCAF is not going to be a trivial consequence of the present paper, it will illustrate the approach that must be taken. We leave this issue for future research.

The recurrent process in Step 1 of our conditioning is closely related the {\em bang-bang process} constructed in Section 5.3 of Evans and Henning \cite{EH}. Therein the authors follow the construction of a resurrected process from Fitzsimmons \cite{Fit91}, which is in fact repeated (unboundedly many) applications  of an $h$-transform using the function $u(\cdot, y)$: construct the $h$-transformed path until the last hitting time of $y$ and then, instead of killing the paths, start a new $h$-transform from point $y$. The resulting process will have the same distribution as the solution of (\ref{e:step1}) in Step 1 of our construction. We follow a different approach and use the construction proposed in \cite{Crt}, which has a  natural interpretation as a locally absolutely continuous Girsanov change of measure. One advantage of this approach is that it directly gives an explicit relationship between the laws of the original process and its recurrent transform (see \ref{e:ACstep1} and \ref{e:survival}), which was useful in proving Theorem \ref{t:Lbrdige} -- one of the main results of this paper. 

Moreover, the construction of the recurrent process using the SDE in (\ref{e:step1}) is more easily implementable in practice. Namely, one can apply a simple Euler scheme using only the knowledge of the drift and diffusion coefficients of the original process. On the other hand, the bang-bang process of Evans and Hening does not admit an easy `adapted' construction. Indeed, one needs to first construct the entire trajectory of the $h$-transform over the infinite time interval $(0,\infty)$ to find its last hitting time of $y$, which is not a stopping time. The construction over the infinite time interval will be needed even if one is only interested in the behaviour of the bang-bang process on, say, $[0,t]$ as one cannot know for sure whether the last hitting time of $y$ is before $t$ or not. Furthermore, this procedure has to be repeated  possibly infinitely many times.   However, leaving the practical difficulties aside, if one is solely interested in an {\em  abstract}  construction of the particular recurrent process of (\ref{e:step1}), one can alternatively use the bang-bang process of Evans and Hening.

A final but rather obvious remark on the type of conditioning considered here is that the same ideas, after minor modifications, can be used to condition $X$ so that its local time at a given point stays below a fixed level at all times. The resulting diffusion will be clearly transient and can be considered as a specific case of the problem of the weak convergence of measures $(P_t^x)_{t \geq 0}$, where $P_t^x:=P^x(\cdot|L^x_s\leq f(s), s\leq t)$ and $f$ is a given function. This problem is studied in detail when $X$ is a Brownian motion by Kolb and Savov \cite{KolbS}. It will be interesting  to see whether the approach developed herein can provide insights  for this problem in case of  general one-dimensional diffusions.

The outline of the paper is as follows. Section \ref{s:recipe} reviews some background material on one-dimensional diffusions that will be used often in the paper  and gives a recipe for the construction of the SDE to achieve our desired conditioning. Section \ref{s:rectr}  finds the drift that will be necessary to complete Step 1. The drift term that will be used in the second step of our construction is discussed in Section \ref{s:excursion} along with its connection to the excursions of $X$. Our main results on the SDE representation of $X$ conditioned on its local time at its lifetime and a new path decomposition result for transient diffusions are contained in Section \ref{s:main}. Finally, Section \ref{s:examples} illustrates the findings via some specific examples.

\section{Preliminaries and a recipe for the conditioning} \label{s:recipe}
Let $X$ be a regular transient diffusion on $(l,r)$, where $ -\infty \leq l <r \leq \infty$. Such a diffusion is uniquely characterised by its scale function, $s$, and speed measure, $m$, defined on the Borel subsets of the open interval $(l,r)$. We assume that if any of the boundaries are reached in finite time, the process is killed and sent to the cemetery state, $\Delta$. This is the only instance when the process can be killed, we do not allow killing inside $(l,r)$. Consistent with the term `killing' $\Delta$ is assumed to be an absorbing state. The set of points that can be reached in finite time starting from the interior of $(l,r)$ and the entrance boundaries will be denoted by $I$. That is, $I$ is the union of $(l,r)$ with the regular, exit or entrance boundaries.  The law induced on $C(\bbR_+,I)$,  the space of $I$-valued continuous functions on $[0,\infty)$, by $X$ with $X_0=x$ will be denoted by $P^x$ as usual, while $\zeta$ will correspond to its lifetime. In what follows we will often replace $\zeta$ with $\infty$ when dealing with the limit values of the processes as long as no confusion arises. The filtration $(\cF_t)_{t \geq 0}$ will correspond to the universal completion of the natural filtration of $X$, and therefore is right continuous. Recall that in terms of the first hitting times,   $T_y:=\inf\{t> 0: X_t=y\}$ for $y \in (l,r)$, the regularity amounts to $P^x(T_y<\infty)>0$ whenever $x$ and $y$ belongs to the open interval $(l,r)$. This assumption entails in particular that $s$ is strictly increasing and continuous (see Proposition VII.3.2 in \cite{RY}),  on $(l,r)$, and $0<m((a,z)<\infty$ for all $l<a<z<r$ (see Theorem VII3.6 and the preceding discussion in \cite{RY}). 

The hypothesis that $X$ is transient implies that at least one of $s(l)$ and $s(r)$ must be finite. Since $s$ is unique only upto an affine transformation, we will assume without loss of generality that $s(l)=0$ and  $s(r)=1$ whenever they are finite. In view of our foregoing assumptions one can easily deduce that $X_{\zeta -}\in \{l,r\}$. We refer the reader to \cite{BorSal} for a summary of results and references on one-dimensional diffusions. The definitive treatment of such diffusions is, of course, contained in \cite{IM}. 

 As we are interested in the path transformations of local times via SDEs, we further impose the so-called {\em Engelbert-Schmidt conditions}  to ensure that  $X$ can be considered as a solution of an SDE. That is, we shall assume the existence of measurable functions $\sigma:(l,r)\mapsto \bbR$ and $b: (l,r)\mapsto \bbR$ such that 
\be \label{e:ESR}
\sigma(x) >0  \mbox{ and } \exists \eps >0 \mbox{ s.t. } \int_{x-\eps}^{x+\eps}\frac{1+ |b(y)|}{\sigma^2(y)}dy <\infty \mbox{ for any $x \in (l,r)$.}
\ee
Under this assumption (see \cite{ES} or Theorem 5.5.15 in \cite{KS}) there exists a {\em unique} weak solution (upto the exit time from the interval $(l,r)$) to the SDE
\[
X_t=x+\int_0^t\sigma(X_s)dB_s + \int_0^t b(X_s)ds, \qquad t < \zeta,
\]
where $\zeta=\inf\{t\geq 0: X_{t-}\in \{l,r\}\}$ and $l<x<r$.  
Moreover,  the condition (\ref{e:ESR}) further implies one can take
\be \label{e:scalespeed}
s(x)=\int_C^x \exp\left(-2 \int_c^z \frac{b(u)}{\sigma^2(u)}du\right)dz
\; \mbox{ and }\;  m(dx)=\frac{2}{s'(x)\sigma^2(x)}dx, \mbox{ for some } (c,C) \in (l,r)^2.
\ee

The Engelbert-Schmidt conditions ensure the existence and uniqueness of weak solutions to above SDE starting from $x \in (l,r)$. However, it should be noted that the results of Engelbert and Schmidt do not apply if the starting point is an entrance boundary since the scale function is not finite at such an endpoint.  The following theorem, whose proof is delegated to the Appendix, extends this theory when $x$  is an entrance boundary for a transient diffusion. A quick glance at the proof reveals that the solution is obtained as a time and scale change of a 3-dimensional Bessel process starting from its entrance boundary as opposed to the standard Engelbert-Schmidt theory that constructs the solution as a time and scale change of a Brownian motion.  
\begin{theorem} \label{t:SDEent} Suppose that $X$ is a regular transient diffusion on $(l,r)$ such that its scale function and speed measure are defined by (\ref{e:scalespeed}), where $\sigma:(l,r)\mapsto \bbR$ and $b: (l,r)\mapsto \bbR$   are measurable functions satisfiying (\ref{e:ESR}).  Assume further that $X$ has  an entrance boundary. Then there exists a unique weak solution to 
\be \label{e:SDEent}
X_t=x+\int_0^t\sigma(X_s)dB_s + \int_0^t b(X_s)ds, \qquad t < \zeta,
\ee
where $\zeta=\inf\{t> 0: X_{t-}\in \{l,r\}\}$ and $x$ is the entrance boundary\footnote{The fact that $X$ is transient implies there exists at most one entrance boundary.}.
\end{theorem}

We summarise the assumptions on $X$ in the following
\begin{assumption}\label{a:X}  $X$ is a regular transient diffusion on $(l,r)$, where $ -\infty \leq l <r \leq \infty$, with no killing inside $(l,r)$. Moreover, whenever $X_0$ is an entrance boundary or belongs to $(l,r)$, $X$ is the unique weak solution to 
\be \label{e:sdeR}
X_t=X_0+\int_0^t\sigma(X_s)dB_s + \int_0^t b(X_s)ds, \qquad t < \zeta,
\ee
where $\zeta=\inf\{t>0: X_{t-}\in \{l,r\}\}$, and  $\sigma:(l,r)\mapsto \bbR$ and $b: (l,r)\mapsto \bbR$ are measurable functions satisfying (\ref{e:ESR}).  Its scale function is chosen so that  $s(l)=0$ and  $s(r)=1$ whenever they are finite.
\end{assumption}

We will denote by $(\tilde{L}_t^x)_{x \in (l,r)}$  the family of diffusion local times associated to $X$. Recall that the occupation times formula for the diffusion local time is given by
\[
\int_0^t f(X_s)ds =\int_{l}^{r}f(x)\tilde{L}^x_t m(dx).
\]
From above one can easily deduce the relationship a.s..
\[
L^x= \frac{2}{s'(x)}\tilde{L}^x,
\]
where $L^x$ is the semimartingale local time of $X$ at $x$ defined by 
\[
\int_0^t f(X_s)\sigma^2(X_s)ds =\int_{l}^{r}f(x)L^x_t dx.
\]
\begin{remark}
It is clear from the above relationship that the diffusion local time is not invariant under absolutely continuous changes of measures, which change the scale function. This distinction will be important in the sequel when we consider absolutely continuous changes of measure to obtain the SDEs that achieve the conditioning we are after. Note that  the semimartingale local time can be defined as a limit involving the quadratic variation of $X$ (see Corollary VI.1.9 in \cite{RY}), which remains intact after an absolutely continuous measure change.   
\end{remark}
Any regular transient diffusion on $(l,r)$ has a finite potential density $u: (l,r)^2 \mapsto \bbR_+$, where $\bbR_+:=[0,\infty)$, with respect to its speed measure (see p.20 of \cite{BorSal}). That is, for any bounded and continuous function, $f$, vanishing at accessible boundaries
\[
Uf(x):=\int_0^{\infty}E^x[f(X_t)]dt=\int_l^r f(y)u(x,y)m(dy).
\]

In the case of one-dimensional transient diffusions that we consider herein the distribution of $L^y_{\infty}$ is known explicitly in terms of the potential density (see p.21 of \cite{BorSal}). In particular, 
\be \label{e:lawLT}
P^y(L^y_{\infty}>t)=P^y\left(\tilde{L}^y_{\infty}>\frac{s'(y)t}{2}\right)=\exp\left(-\frac{s'(y)t}{2u(y,y)}\right).
\ee
Therefore, for an arbitrary starting point, $x$, in $(l,r)$ and any Borel subset, $E$, of $\bbR_+$ we have
\be \label{e:L_dist}
P^x(L^y_{\infty}\in E|\cF_t)=\chf_{[L^y_t \in E]}(1-\psi(X_t,y))+ \frac{s'(y)\psi(X_t,y)}{2 u(y,y)}\int_E \chf_{[a>L^y_t]}\exp\left(-\frac{s'(y)(a-L^y_t)}{2u(y,y)}\right)da,
\ee
where
\[
\psi(x,y):=P^x(T_y<\infty),
\]
in view of the strong Markov property of $X$. 

It will also prove useful in the sequel to consider  the {\em inverse local time}:
\[
\tau^y_a =\inf\{t\geq 0: L^y_t >a\}.
\]
$(\tau^y_a)_{a \geq 0}$ is right continuous and, moreover,
\[
\tau^y_{a-}=\inf\{t\geq 0: L^y_t \geq a\}.
\]
Clearly, the interval $[\tau^y_{a-},\tau^y_{a}]$ corresponds to an interval of constancy for $L^y$ or, equivalently, an excursion of $X$ from the point $y$. 

When only one of $s(l)$ and $s(r)$ is finite, the terminal value of $X$ equals either $l$ or $r$, in which case $X_{\infty}$ and $L^y_{\infty}$ are trivially independent no matter where the diffusion has started. If both $s(l)$ and $s(r)$ are finite, the situation is more delicate. The following result that illustrates this must be well-known. We nevertheless provide a proof for the convenience of the reader.
\begin{proposition} \label{p:LTRindep} Suppose that $X$ is a regular transient diffusion on $(l,r)$ with $s(l)=0=1-s(r)$. Then, $X_{\infty}$ and $L^y_{\infty}$ are independent under $P^x$ if and only if $x=y$.
\end{proposition}
\begin{proof}
Observe that the independence of $X_{\infty}$ and $L^y_{\infty}$ is equivalent to the independence of $X_{\infty}$ and $\tilde{L}^y_{\infty}$. Suppose  that $x \leq y$ and observe from (\ref{e:L_dist}) that 
\[
P^x(\tilde{L}^y_{\infty}>t)=P^x(T_y <\infty) \exp\left(-\frac{t}{u(y,y)}\right).
\]
Next, consider the $h$-transform of $X$ using $s$ as the $h$-function to obtain a diffusion with transition function
\[
P^s_t(x,dz)= \frac{s(z)}{s(x)}P_t(x,dy),
\]
where $P_t$ is the original transition function of $X$. The resulting process corresponds to the conditioning of $X$ so that $X_{\infty}=r$ and it is a regular diffusion on $(l,r)$ with the speed measure, $m^s$, and the potential density\footnote{Note that this is the density with respect to the new speed measure $m^h$.} $u^s$ given by (see Paragraph 31 in Chap. II of \cite{BorSal})
\[
m^s(dz)=s^2(z)m(dz)\qquad u^s(x,z)=\frac{u(x,z)}{s(x)s(z)}.
\]
Denoting the law of the $h$-transform with $P^{s,x}$ when it starts at $x$, we have, in particular,
\[
P^{s,x}(\tilde{L}^y_{\infty}>t)=P^x(\tilde{L}^y_{\infty}>t|X_{\infty}=r).
\]
Notice that, since the speed measure has changed, the diffusion local time of the $h$-transform is no longer represented by $\tilde{L}^y$. Let $\tilde{L}^{s,x}$ denote the diffusion local time with respect to $m^s$ at level $x$. In view of the occupation times formula and the fact that $P^{s,x}\sim P^x$ on $\cF_t$ one has
\[
\int_{l}^{r}f(x)\tilde{L}^x_t m(dx)=\int_0^t f(X_s)ds =\int_{l}^{r}f(x)\tilde{L}^{s,x}_t s^2(x)m(dx), \quad P^{s,x}\mbox{-a.s.},
\]
which yields $P^{s,x}\mbox{-a.s.}$ $\frac{\tilde{L}^y_{t}}{s^2(y)}=\tilde{L}^{s,y}_{t}$ for all $t$ and $y$ due to the joint continuity of diffusion local times. Since $\tilde{L}^y_t$ increases to $\tilde{L}^y_{\infty}$ under $P^x$ as $t \rar \infty$, so it does under $P^{s,x}$ since $P^{s,x}\ll P^x$. Moreover, $\tilde{L}^{s,y}_t$ increases to $\tilde{L}^{s,y}_{\infty}$ under $P^{s,x}$ as $t \rar \infty$, too.  Therefore, $\frac{\tilde{L}^y_{\infty}}{s^2(y)}=\tilde{L}^{s,y}_{\infty}, \, P^{s,x}$-a.s. and
\[
P^x(\tilde{L}^y_{\infty}>t|X_{\infty}=r)=P^{s,x}\left(\tilde{L}^{s,y}_{\infty}>\frac{t}{s^2(y)}\right)=\exp\left(-\frac{t}{s^2 u^s(y,y)}\right)=\exp\left(-\frac{t}{u(y,y)}\right),
\]
since $P^{s,x}(T_y<\infty)=1$ due to the conditioning. Thus,
\[
P^x(\tilde{L}^y_{\infty}>t|X_{\infty}=r)=P^x(\tilde{L}^y_{\infty}>t)
\]
if and only if $x=y$. The case $x\geq y$ is handled similarly by conditioning on $[X_{\infty}=l]$ using the $h$-function $1-s$.
\end{proof}
Note that if $s(l)=0=1-s(r)$, then $P^x(X_{\infty}=r)=s(r)$ and
\be \label{e:psiff}
\psi(x,y)= \left\{\ba{rl}
\frac{s(x)}{s(y)}, & y\geq x; \\
\frac{1-s(x)}{1-s(y)}, & y < x.
\ea \right . \qquad \qquad u(x,y)= s(x)(1-s(y)), \; x \leq y.
\ee
On the other hand, if $s(l)=0$ and $s(r)=\infty$, then $X_t \rar l$, $P^x$-a.s. for any $x \in (l,r)$, which in turn implies
\be \label{e:psifi}
\psi(x,y)= \left\{\ba{rl}
\frac{s(x)}{s(y)}, & y\geq x; \\
1, & y < x.
\ea \right . \qquad \qquad u(x,y)= s(x), \; x \leq y.
\ee
Similarly, if $s(l)=-\infty$ and $s(r)=1$, then $X_t \rar r$, $P^x$-a.s. for any $x \in (l,r)$, and 
\be \label{e:psiif}
\psi(x,y)= \left\{\ba{rl}
1, & y\geq x; \\
\frac{1-s(x)}{1-s(y)}, & y < x.
\ea \right . \qquad \qquad u(x,y)= 1-s(y), \; x \leq y.
\ee

For later purposes we also define 
\be \label{e:Pexitr}
\rho(y):=P^y(X_{\infty}=r).
\ee
We are interested in conditioning $X$ so that $L^y_{\infty}=a$ for some given $a>0$. As any such conditioning will make sure that $X$ first hits $y$, we will assume $X_0=y$ to ease the exposition. Formally, the construction of the conditioned process should be achieved in two steps: 1) make sure that $X$ keeps hitting $y$ before $L^y$ reaches $a$ and 2) as soon as $L^y$ becomes $a$  never let $X$ hit $y$ again. 

In order to achieve the first step we need to change the behaviour of $X$ in such a way that the process is {\em recurrent}. Indeed, if $X$ is still transient after some transformation, there will be a positive probably that it will drift towards one of its endpoints before $L^y$ becomes $a$.  Section \ref{s:rectr} uses  a particular { recurrent transformation}  to complete the first step of our conditioning. 

The second step in our recipe is to prevent $X$ from hitting $y$ after $\tau^y_{a-}$.  Since $X_{\tau^y_{a-}}=y$, on $[\tau^y_{a-}<\infty]$, this means that we need to keep $X$ above or below $y$ after $\tau^y_{a-}$.  Recall that we are not merely interested in creating a process with the property that $L^y_{\infty}=a$,  but a conditioned version of $X$ whose law coincides with the regular conditional probability $P^y(\cdot | L^y_{\infty}=a)$. This necessitates, in particular, that the conditioned process should also have the same set of possible values for its limiting value.  If $s(l)=0$ and $s(r)=\infty$ (resp.  $s(l)=-\infty$ and $s(r)=1$),  our task is relatively simple: keep $X$ below (resp. above) $y$ at all times after $\tau^y_{a-}$. 

On the other hand, if $s(l)=0=1-s(r)$, the original process could drift towards $r$ as well as $l$.  As we are only conditioning on $L^y_{\infty}$ and not on $X_{\infty}$, we will have to appropriately {\em randomise} the coefficients of the SDE for the bridge process to allow our solution have $l$ and $r$ as possible limit points. In Section \ref{s:excursion} we study the {\em Bessel-type} SDEs with appropriate  drifts that will allow us  to complete the second step, discuss its connection with the excursions of $X$, and present a time reversal connection. Finally, the SDE that achieves our desired conditioning with randomised drifts is constructed in Section \ref{s:main}. 

\section{A recurrent transformations and Step 1} \label{s:rectr}
The first step towards our desired conditioning requires us transform $X$ into a recurrent diffusion. To do so we shall resort to a recurrent transform as developed in \cite{Crt}.  In order to recall the concept suppose that $h$ is a non-negative $C^2$-function and $M$ an adapted continuous process of finite variation so that $h(X)M$ is a non-negative local martingale. Thus, by stopping at its localising sequence and using Girsanov's theorem we arrive at a weak solution, up to a stopping time, of the following equation for any given $x \in (l,r)$ taking values in $(l,r)$:
\be \label{e:sderectr}
X_t=x+ \int_0^t \sigma(X_s)dB_s +\int_0^t \left\{b(X_s) + \sigma^2(X_s)\frac{h'(X_s)}{h(X_s)}\right\}ds.
\ee
We can associate to this SDE the scale function 
\be \label{e:recscale}
r_s(x):=\int_c^x \frac{s'(y)}{h^2(y)}dy, \qquad x \in (l,r)
\ee
provided that the integral is finite for all $x \in (l,r)$, which, in particular, requires $h>0$ on $(l,r)$. This will allow us to deduce the existence of a solution to (\ref{e:sderectr}) until the first exit time from $(l,r)$. If, in addition, $-r_s(l+)=r_s(r-)=\infty$, the solution will be recurrent and never exit $(l,r)$ (see Proposition 5.5.22 in \cite{KS}). 

\begin{remark} It has to be noted that the notion of recurrence that we consider for one-dimensional diffusions excludes some recurrent solutions of one-dimensional SDEs with time-homogeneous coefficients since we kill our diffusion as soon as it reaches a boundary point. A notable example is a squared Bessel process with  dimension $\delta <2$, which solves the following SDE:
\[
X_t= x+2 \int_0^t\sqrt{X_s}dB_s +\delta t.
\]
The above SDE has a global strong solution, i.e. solution for all $t \geq 0$, which is recurrent (see Section XI.1 of \cite{RY}). However, the point $0$ is reached a.s. and is instantaneously reflecting by Proposition XI.1.5 in  \cite{RY}. As such, it violates our assumption of a diffusion being killed at a regular boundary. According to  this assumption, a squared Bessel process of dimension $0<\delta<2$ is  killed as soon as it reaches $0$ and, thus, is a {\em transient} diffusion. 
\end{remark}

The following definition is borrowed from \cite{Crt}:
\begin{definition} Let $X$ be a regular transient diffusion satisfying Assumption \ref{a:X} and  $h:(l,r)\mapsto (0,\infty)$ be a continuous function such that the limits $h(l+):=\lim_{x \rar l}h(x)$ and $h(r-):=\lim_{x \rar r}h(x)$ exists. Then, $(h,M)$ is said to be a recurrent transform (of $X$) if the following are satisfied:
\begin{enumerate}
\item $M$ is an adapted process of finite variation.
\item $h(X)M$ is a nonnegative local martingale. 
\item The function $r_s$ from (\ref{e:recscale}) is finite for all $x \in (l,r)$ with $-r_s(l+)=r_s(r-)=\infty$.
\item There exists a unique weak solution to (\ref{e:sderectr}) for $t \geq 0$ for any $x \in (l,r)$.
\end{enumerate}
\end{definition}

The following proposition, which  will let us achieve the first step of our desired conditioning, is proved in \cite{Crt}. However, we reproduce the statement and its proof for the convenience of the reader.
\begin{proposition} \label{p:step1} Let $y\in (l,r)$ be fixed as in the previous section and consider the pair $(h,M)$ defined by
\[
h(x):=u(x,y), \; x \in (l,r),\mbox{ and } M_t=\exp\left(\frac{s'(y)L^y_t}{2u(y,y)}\right).
\]
Then, the following hold:
\begin{enumerate}
\item There exists a unique weak solution to 
\be \label{e:step1}
X_t=x+\int_0^t\sigma(X_s)dB_s + \int_0^t \left\{b(X_s)+\sigma^2(X_s)\frac{u_x(X_s, y)}{u(X_s, y)}\right\} ds, \qquad t \geq 0,
\ee
for any $x \in (l,r)$, where  $u_x$ denotes the first partial left derivative of $u(x,y)$ with respect to $x$. 
\item $(h,M)$ is a recurrent transform for $X$.
\item Moreover, If $R^{h,x}$ denotes the law of the solution, then, for all $a>0$, we have  $R^{h,x}(L^y_{\infty} \geq a)=R^{h,x}(\tau^y_{a-} <\infty)=1$ and
\be \label{e:ACstep1}
\frac{dR^{h,x}}{dP^x}\big|_{\cF_{\tau^y_{a-}}}=\frac{u(y,y)}{u(x,y)}\exp\left(\frac{as'(y)}{2u(y,y)}\right)\chf_{[\tau^y_{a-}<\zeta]}.
\ee
In general if $T$ is a stopping time such that $R^{h,x}(T<\infty)=1$, then for any $ F \in \cF_T$ the following identity holds:
\be \label{e:survival}
P^x(\zeta>T, F)= u(x,y)E^{h,x}\left[\chf_F \frac{1}{u(X_T,y)}\exp\left(-\frac{s'(y)}{2u(y,y)}L^y_T\right)\right],
\ee
where $E^{h,x}$ is the expectation operator with respect to the probability measure $R^{h,x}$.
\end{enumerate}
\end{proposition}
\begin{proof}
\begin{enumerate}
\item It follows from  (\ref{e:psiff})-(\ref{e:psiif}) that $x \mapsto u(x,y)$ is stricly positive,absolutely continuous and its left derivative is of finite variation. In fact, for each $y \in (l,r)$ $x \to u(x,y)$ is differentiable at all $x \neq y$. This implies that the Engelbert-Schmidt conditions are satisfied for  the SDE in (\ref{e:step1}). To see this recall that  $\sigma$ and $b$  satisfy (\ref{e:ESR}) by assumption. Thus, one only needs to check, for any $x\in (l,r)$, there exists an $\eps>0$ such that
\be \label{e:escRT}
\int_{x-\eps}^{x+\eps}\frac{|u_x(z,y)|}{u(x,z)}dz<\infty.
\ee
Indeed, if $x <y$, $u_x(x,y)>0$ and for any $\eps>0$ such that $x+\eps<y$, one has
\[
\int_{x-\eps}^{x+\eps}\frac{|u_x(z,y)|}{u(x,z)}dz=\int_{x-\eps}^{x+\eps}\frac{u_x(z,y)}{u(x,z)}dz=\log\frac{u(x+\eps,y)}{u(x-\eps,y)}<\infty.
\]
Similarly, one can show that (\ref{e:escRT}) is satisfied for some $\eps>0$ when $x>y$ since $u(\cdot,y)$ is decreasing on $(y,r)$. Moreover,
\[
\int_{y-\eps}^{y+\eps}\frac{|u_x(z,y)|}{u(x,z)}dz=\log\frac{u(y,y)}{u(y-\eps,y)}+\log\frac{u(y,y)}{u(y+\eps,y)} <\infty.
\]
 Thus, there exists a unique weak solution to (\ref{e:step1}) by Theorem 5.5.15 in \cite{KS} upto the exit time from the interval $(l,r)$. 
 
 We shall next show  that  solutions of (\ref{e:step1}) never hit $l$ or $r$ in finite time, implying in particular their recurrent behaviour. This  will follow if the scale function of (\ref{e:step1}) is unbounded near $l$ and $r$.

To this end consider  the function
\[
r_s(x)=\int_y^x \frac{s'(z)}{u^2(z,y)}dz,
\]
and suppose, first, that $s(l)=1-s(r)=0$. Then, for $x <y$,
\[
r_s(x)=\frac{1}{(1-s(y))^2}\int_y^x\frac{s'(z)}{s^2(z)}dz=\frac{1}{(1-s(y))^2}\left(\frac{1}{s(y)}-\frac{1}{s(x)}\right),
\]
which in particular shows that $\lim_{x \rar l}r_s(x)=-\infty$. Similarly, for $x >y$,
\[
r_s(x)=\frac{1}{s^2(y))}\int_y^x\frac{s'(z)}{(1-s(z))^2}dz=\frac{1}{s^2(y)}\left(\frac{1}{1-s(x)}-\frac{1}{1-s(y)}\right),
\]
and, thus, $s_r(\infty)=\infty$. The other cases are handled  the same way to show $-r_s(l)=r_s(r)=\infty$. 
\item It follows from a simple application of Ito-Tanaka formula that $h(X)M$ is a nonnegative local martingale. Thus, computations of the previous part yields $(h, M)$ is  a recurrent transformation.
\item First note that if $F \in \cF_T$ for some $(\cF_t)$-stopping time, $T$, such that $h(X^T)M^T$ is a uniformly integrable $P^y$-martingale,  
 \be \label{e:ACrectr}
 R^{h,x}(F)=\frac{1}{h(x)}E^{x}\left[\chf_{F}h(X_T)M_T\right].
 \ee
Since $h(X^{\tau^y_{a-}})M^{\tau^y_{a-}}$ is a bounded martingale,  it follows from (\ref{e:ACrectr}) that 
\bean
R^{h,x}(\tau^y_{a-}<\zeta)&=&\frac{1}{u(x,y)}E^x\left[\chf_{[\tau^y_{a-}<\zeta]}u(X_{ \tau^y_{a-}},y)\exp\left(\frac{s'(y)L^y_{\tau^y_{a-}}}{2u(y,y)}\right)\right]\\
&=&\frac{u(y,y)}{u(x,y)}\exp\left(\frac{a s'(y)}{2 u(y,y)}\right)P^x(\tau^y_{a-}<\zeta),
\eean
where the last equality is due to the fact that $X_{\tau^y_{a-}}=y$ and $L^y_{\tau^y_{a-}}=a$ on $[\tau^y_{a-}<\zeta]$ by the continuity of $X$ and of $L^y$. On the other hand,
\[
P^x(\tau^y_{a-}<\zeta)=P^x(L^y_{\infty}\geq a)=\psi(x,y) \exp\left(-\frac{as'(y)}{2u(y,y)}\right)
\]
in view of (\ref{e:L_dist}). Since $\psi(x,y)=\frac{u(x,y)}{u(y,y)}$ and, $R^{h,y}(\zeta=\infty)=1$, we deduce $R^{h,y}(\tau^y_{a-}<\infty)=R^{h,y}(L^y_{\infty}\geq a)=1$. Applying (\ref{e:ACrectr}) once more yields the desired absolute continuity on $\cF_{\tau^y_{a-}}$.

To show the remaining assertion let $T$ be $R^{h,x}$-a.s. finite stopping time. Then, 
\bean
&&E^{h,x}\left[\chf_F\chf_{[L^y_T\leq a]} \frac{1}{u(X_T,y)}\exp\left(-\frac{s'(y)}{2u(y,y)}L^y_T\right)\right]\\
&=&E^x\left[\chf_F\chf_{[L^y_T\leq a]} \frac{1}{u(X_T,y)}\exp\left(-\frac{s'(y)}{2u(y,y)}L^y_T\right)\frac{u(y,y)}{u(x,y)}\exp\left(\frac{a s'(y)}{2 u(y,y)}\right)\chf_{[\tau^y_{a-}<\zeta]}\right]\\
&=&E^x\left[\chf_F\chf_{[L^y_T\leq a]} \frac{1}{u(X_T,y)}\exp\left(-\frac{s'(y)}{2u(y,y)}L^y_T\right)\frac{u(X_{\tau^y_{a-}},y)}{u(x,y)}\exp\left(\frac{s'(y)}{2 u(y,y)}L^y_{\tau^y_{a-}}\right)\chf_{[\tau^y_{a-}<\zeta]}\right]\\
&=&\frac{1}{u(x,y)}E^x\left[\chf_F\chf_{[L^y_T\leq a]} \chf_{[T<\zeta]}\right],
\eean
where the first equality is a consequence of the above absolute continuity relationship, the second equality is due to the fact that $X_{\tau^y_{a-}}=y$ on $[\tau^y_{a-}<\zeta]$, and the last line follows from $h(X^{\tau^y_{a-}})M^{\tau^y_{a-}}$ being a bounded $P^x$-martingale as observed above, 

Next note that $L^y_T<\infty$, $R^{h,x}$-a.s. since $R^{h,x}(T<\infty)=1$, and $t \mapsto L^y_t$ is $R^{h,x}$-a.s. continuous on $(0,\infty)$. Moreover, $P^x(L^y_T <\infty)=1$ as well since $L^y_T \leq L^y_{\zeta}<\infty, \, P^x$-a.s. for $X$ is a transient diffusion under $P^x$. Therefore, letting $a \rar \infty$ the claim follows from the monotone convergence theorem. 
\end{enumerate}
\end{proof}

Proposition \ref{p:step1} tells us what to do in our first step: We run the $(h,M)$-recurrent transformation given in the proposition until $\tau^y_{a-}$, which is finite with probability 1. That is,
\[
X_{t\wedge \tau^y_{a-}}=y+\int_0^{t\wedge\tau^y_{a-}}\sigma(X_s)dB_s + \int_0^{t\wedge\tau^y_{a-}} \left\{b(X_s)+\sigma^2(X_s)\frac{u_x(X_s, y)}{u(X_s, y)}\right\} ds.
\]
Recall that $[\tau^y_{a-}<\infty]=[L^y_{\infty}\geq a]$. Thus, the above makes sure that the conditioned process will have its local time at $y$ being at least equal to $a$ in the limit. It now remains to make sure that the process never visits $y$ after $\tau^y_{a-}$
\section{Bessel-type motions and Step 2} \label{s:excursion}
Recall that the second step of our recipe is to keep $X$ away from $y$ after $\tau^y_{a-}$. As one can guess this can be achieved by conditioning $X$ to never hit $y$ using an $h$-transform.  The next proposition make this idea rigorous and  gives us the candidate drifts that ensure $X$ has the prescribed limit while avoiding $y$ at the same time.

\begin{proposition} \label{p:step2} Let $X$ be a regular diffusion satisfying Assumption \ref{a:X}. 
\begin{enumerate} 
\item Suppose $s(l)=0$. There exists a regular diffusion on $(l,y)$ with the scale function $s_l$ and the speed measure, $m_l$, defined by
\[
s_l(x):=\frac{1}{s(y)-s(x)}, \qquad m_l(dx):=(s(y)-s(x))^2 m(dx). 
\]
$y$ is an entrance boundary for this diffusion, which is also the unique weak solution to 
\be \label{e:sdel}
R_t= x+ \int_0^t\sigma(R_s)dB_s +\int_0^t \left\{b(R_s) - \frac{s'(R_s)\sigma^2(R_s)}{s(y)-s(R_s)}\right\}ds, \qquad t<\zeta, 
\ee
where $x\in (l,y]$ and $\zeta=\inf\{t\geq 0: R_{t-}=l\}$. Moreover, $\lim_{t \rar \infty}R_t=l,\, Q^{x,0}$-a.s., where $Q^{x,0}$ is the law of the weak solution to the SDE above.
\item Suppose $s(r)=1$. There exists a regular diffusion on $(y,r)$ with the scale function $s_r$ and the speed measure, $m_r$, defined by
\[
s_r(x):=\frac{1}{s(y)-s(x)}, \qquad m_r(dx):=(s(y)-s(x))^2 m(dx). 
\]
$y$ is an entrance boundary for this diffusion, which is also the unique weak solution to 
\be \label{e:sder}
R_t= x+ \int_0^t\sigma(R_s)dB_s +\int_0^t \left\{b(R_s) +\frac{s'(R_s)\sigma^2(R_s)}{s(R_s)-s(y)}\right\}ds, \qquad t<\zeta, 
\ee
where $x\in [y,r)$ and $\zeta=\inf\{t\geq 0: R_{t-}=r\}$. Moreover, $\lim_{t \rar \infty}R_t=r,\, Q^{x,1}$-a.s., where $Q^{x,1}$ is the law of the weak solution to the SDE above.
\end{enumerate}
\end{proposition}
\begin{proof}
We will only prove the proposition in case (1), the proof of the second part follows identical steps. 

Clearly, $s_l$ is strictly increasing and continuous since $s'>0$ on $(l,y)$. Moreover, it can be directly checked that $0<m_l((a,z))<\infty$ for $l<a<z<y$ using the analogous property for $m$ and the fact that $s$ is strictly increasing. Thus, we can associate a regular diffusion to $(s_l,m_l)$.  To see that $y$ is an entrance boundary, observe that for $l<z<y$
\[
\int_z^y (s_l(a)-s_l(z))(s(y)-s(a))^2m(da)<\infty.
\]
Indeed, since $s_l(y-)=\infty$,
\[
\lim_{a \rar y} \frac{s_l(a)-s_l(z)}{(s(y)-s(a))^{-2}}=\half \lim_{a \rar y} \frac{\frac{s'(a)}{(s(y)-s(a))^2}}{s'(a)(s(y)-s(a))^{-3}}=0,
\]
implying the finiteness of the above integral near $y$ due to the fact that $y$ belongs to the interior of the state space of the regular diffusion with scale $s$ and speed $m$. Finiteness of the integral near $z$ follows from the boundedness of $s_l$ and $s$ on the compact subsets of $(l,y)$ and the finiteness of $m$ in the interior of $(l,r)$. 

In order to conclude $y$ is an entrance boundary, we also need to verify that 
\[
\int_z^y m_l((z,a))\frac{s'(a)}{(s(y)-s(a))^2}da=\infty.
\]
The above integral diverges since $m_l((z,y))>0$, and 
\[
\int_{z'}^y \frac{s'(a)}{(s(y)-s(a))^2}da=\infty
\]
for any $z'$ such that $z<z'<y$.

In order to show the weak existence and uniqueness of  solutions to (\ref{e:sdel}) with $x \in (l,y)$, we will again make use of the Engelbert-Schmidt criteria analogous to (\ref{e:ESR}). Observe that in view of our assumption on $\sigma$ and $b$, all that we need to check is the local integrability of $\frac{s'}{s(y)-s}$.

Indeed, for any $(x,z)$ such that $l<x<z<y$
\[
\int_x^z \frac{s'(u)}{s(y)-s(u)}du=\log\frac{s(y)-s(z)}{s(y)-s(x)}<\infty.
\]

The existence and uniqueness of a weak solution when the starting point is $y$, i.e. the entrance boundary, follows from Theorem \ref{t:SDEent}.

Finally, the limit as $t\rar \infty$ follows from the fact that $s_l(l)<\infty$ while $s_l(y)=\infty$ (see Proposition 5.5.22 in \cite{KS}).
\end{proof}

 We will next prove a result which shows that the processes obtained above can be considered as  analogues of $3$-dimensional Bessel process, which is the killed Brownian motion on $(0,\infty)$ conditioned to converge to $\infty$. As such, they define the entrance laws for the excursions of the original process, $X$, away from $y$ as shown by Pitman and Yor (see Section 3 of \cite{PY-DBB}) and can be considered as an excursion of $X$ away from $y$ conditioned to last forever. Moreover, a time reversal relationship exists between the solutions of (\ref{e:sdel}) and those of (\ref{e:sdeR}) stopped at $y$ akin to the one between the $3$-dimensional Bessel process and the killed Brownian motion established by Williams \cite{W73}.  This relationship will be proved by a time reversal result of Nagasawa \cite{Nagasawa}. The following version of this result is taken from Sharpe \cite{Sharpe}. 
\begin{theorem}[Nagasawa \cite{Nagasawa}] \label{t:timereversal} Let $X$ and $\hat{X}$ be standard Markov processes in duality on their common state space with respect to a $\sigma$-finite measure, $\xi$. Let $u(x,y)$ denote the potential kernel density with respect to $\xi$ so that for any positive and measurable $f$
\[
E^x\int_0^{\infty}f(X_t)dt= \int u(x,z) f(z)\xi(dz).
\]
Let $G$ be a co-optional time and define
\[
\tilde{X_t}=\left\{ \ba{ll} 
X_{(G-t)-}, & \mbox{ on } [G<\infty] \mbox{ if } 0<t < G; \\
\Delta, & \mbox{ otherwise}.
\ea
\right.
\]
Fix an initial law $\lambda$ and let $v(x)=\int u(x,z)\lambda(dz)$. Then, under $P^{\lambda}$, the process $\tilde{X}$ is a homogeneous Markov process with transition semigroup, $(\tilde{P}_t)$ defined by
\[
\tilde{P}_tf(x)= \left\{ \ba{ll} 
\frac{P_t fv(x)}{v(x)}, & \mbox{ if }  0<v(x) < \infty; \\
0, & \mbox{ otherwise}.
\ea
\right.
\]
\end{theorem}

Now, we can state and prove the results announced in the paragraph preceding the above theorem.
\begin{proposition} Let $X$ be a  regular diffusion satisfiying Assumption \ref{a:X}, $y \in (l,r)$, and denote by $X^0$ (resp. $X^1$) the killed diffusion process on $(l,y)$ (resp. $(y,r)$) with the scale function $s$ and the speed measure $m$.\footnote{These are simply the solutions of (\ref{e:sdeR}) killed when they reach $y$ or one of $l$ and $r$}
\begin{enumerate}
\item Suppose $s(l)=0$ (resp. $(s(r)=1$).  Then, for any bounded and measurable $f$ and $x \neq y$
\[
Q^0_tf(x)= \frac{P_t^0f(s(y)-s)(x)}{s(y)-s(x)}  \; \left(\mbox{resp. } Q^1_tf(x)= \frac{P_t^1f(s-s(y))(x)}{s(x)-s(y)} \right),
 \]
 where $(Q^0_t)_{t \geq 0}$ (resp. $(Q^1_t)_{t \geq 0}$) is the semigroup associated to  the solutions of (\ref{e:sdel}) (resp. (\ref{e:sder})) while  $(P^0_t)_{t \geq 0}$ (resp.$(P^1_t)_{t \geq 0}$) is the transition semigroup of $X^0$ (resp. $X^1$).
 \item Let $R$ be the solution of (\ref{e:sdel}) (resp. (\ref{e:sder})) with $x=y$. Pick a  $z \in (l,y)$ (resp. $z \in (y,r)$) and define the last passage time
 \[
 G_z:=\sup\{t:R_t=z\}.
 \]
 Next, let $Y$ be the diffusion on $(l,y)$ (resp. $(y,r)$)  obtained by conditioning $X^0$ (resp. $X^1$) converge to $y$ with $Y_0=z$. Then,  the processes 
 \[
 \left\{R_{G_z -t}, 0<t<G_y\right\} \mbox{ and } \left\{Y_t, 0<t<S_y\right\}
 \] have the same law, where 
 \[
 S_y=\inf\{t:Y_t=y\}.
 \]
 In particular, $G_z$ and $S_y$ are finite and have the same distribution.
\end{enumerate}
\end{proposition}
\begin{proof}
We will only prove the above result when $s(l)=0$. The other case is handled in the same way. 
\begin{enumerate} 
\item Suppose $x <y$ and consider the martingale $(s(y)- s(X_{t\wedge T_y})$, where $X$ is a solution of (\ref{e:sdeR}) with $X_0=x$, and $T_y$ is the first passage time of $y$ by $X$. Then, Girsanov's theorem in conjunction with the weak uniqueness of  solutions of (\ref{e:sdel}) yields 
\[
Q^{0}_tf(x)=\frac{E^x\left[f(X_t)\left(s(y)- s(X_{t\wedge T_y})\right)\right]}{s(y)-s(x)}=\frac{E^x\left[f(X_t)\left(s(y)- s(X_t)\right)\chf_{[t<T_y]}))\right]}{s(y)-s(x)},
\]
which establishes the identity for $x <y$. 
\item Since $(Q^0_t)$ is the semigroup of $R$, it is self-dual with respect to $m_l$. Its potential kernel with respect to $m_l$ is symmetric and is given by
\[
u_l(x,z)=s_l(x)-s_l(l)=\frac{s(x)}{s(y)(s(y)-s(x))}, \qquad x \leq z,
\]
since $s_l(y)=\infty$ (see the beginning of p.20 in \cite{BorSal}). Next, define
\[
v(x):=\int_l^r u_l(x,z)\eps_y(dz)=u_l(x,y)=\frac{s(x)}{s(y)(s(y)-s(x))},
\]
where $\eps_y$ is the point mass at $y$. Since $G_z$ is a $Q^{y,0}$-a.s. finite co-optional time, Theorem \ref{t:timereversal} yields that, under $Q^{y,0}$, the transition semigroup of the time-reversed process, $(R_{G_z-t})_{0<t<G_z}$,  is given by
\[
\tilde{P}_tf(x)=\frac{Q_t^0fv(x)}{v(x)}.\]
On the other hand,  it follows from part (1) that
\[
\frac{Q_t^0fv(x)}{v(x)}=\frac{P_t^0fv(s(y)-s)(x)}{v(x)(s(y)-s(x))}=\frac{P_t^0fs(x)}{s(x)},
\]
which establishes the claims.
\end{enumerate}
\end{proof}
In view of the above proposition and following the footsteps of McKean \cite{MKexc},  any solution of (\ref{e:sdel}) (resp. (\ref{e:sder})) will be called a {\em Bessel-type motion} with law $Q^{x,0}$ (resp. $Q^{x,1}$). 
\section{Main results and their proofs} \label{s:main}
We are now ready to prove our main results. 
\begin{theorem} \label{t:Lbrdige} There exists a filtered probability space, $(\Om, \cG, (\cG_t)_{t \geq 0}, P^{L,a})$, which contains a Bernoulli random variable, $\theta$, with\footnote{See (\ref{e:Pexitr}) for the definition of $\rho$.} $P^{L,a}(\theta=1)=\rho(y)=1-P^{L,a}(\theta=0)$ and the adapted pair $(X,B)$ such that  $(\cG_t)_{t \geq 0}$ is right-continuous, $B$ is a standard Brownian motion independent of $\theta$, and $X$ satisfies
\bea
X_t&=&y+ \int_0^t\sigma(X_s)dB_s + \int_0^t b(X_s)ds +\int_0^{t \wedge \tau^y_{a-}}\sigma^2(X_s)\frac{u_x(X_s,y)}{u(X_s,y)}ds \label{e:sdeLbridge}  \\
&&+\int_{t \wedge \tau^y_{a-}}^t\sigma^2(X_s)\left\{\theta \chf_{[X_s> y]}\frac{s'(X_s)}{s(X_s)-s(y)}-(1-\theta)\chf_{[X_s< y]}\frac{s'(X_s)}{s(y)-s(X_s)}\right\}ds, \qquad t <\zeta. \nn
\eea
Moreover, weak uniqueness holds for the solutions of the above SDE. The law induced by its solutions, denoted by $P^{L,a}$ with a slight abuse of notation, on $C(\bbR_+, I)$   satisfies the following properties:
\begin{enumerate}
\item The mapping $a \mapsto P^{L,a}(E)$ is measurable for any Borel subset, $E$, of $C(\bbR_+, I)$ endowed with the locally uniform topology. 
\item  There exists a filtered probability space, $(\tilde{\Om}, \tilde{\cG}, (\tilde{\cG}_t)_{t \geq 0}, P^{L,g})$, which contains a Bernoulli random variable, $\theta$, with $P^{L,g}(\theta=1)=\rho(y)=1-P^{L,g}(\theta=0)$, another $\bbR_{++}$-valued random variable $\Gamma$ with distribution $g$, and the adapted pair $(X,B)$ such that  i) $(\tilde{\cG}_t)_{t \geq 0}$ is right-continuous; ii) $B$ is a standard Brownian motion; iii) $B$, $\theta$ and $\Gamma$ are mutually independent;  and iv) $X$ solves
\bea
X_t&=&y+ \int_0^t\sigma(X_s)dB_s + \int_0^t b(X_s)ds +\int_0^{t \wedge \tau^y_{\Gamma-}}\sigma^2(X_s)\frac{u_x(X_s,y)}{u(X_s,y)}ds  \label{e:sdeLrandbridge}  \\
&&+\int_{t \wedge \tau^y_{\Gamma-}}^t\sigma^2(X_s)\left\{\theta \chf_{[X_s> y]}\frac{s'(X_s)}{s(X_s)-s(y)}-(1-\theta)\chf_{[X_s< y]}\frac{s'(X_s)}{s(y)-s(X_s)}\right\}ds, \qquad t <\zeta. \nn
\eea
Similarly,  the uniqueness in law holds for the solutions of the above SDE with properties i)-iv). Furthermore, denoting the law of its solutions by $P^{L,g}$, we have the following disintegration formula: 
\be \label{e:disintegration}
P^{L,g}=\int_0^{\infty}g(da)P^{L,a}.
\ee
\item $P^{L,a}(L^y_{\infty}=a)=P^{L,g}(L^y_{\infty}=\Gamma)=1$.
\end{enumerate}
\end{theorem}
\begin{proof}
A weak solution can be constructed by the time-change method of Engelbert-Schmidt first applied to a Brownian motion until $\tau^y_{a-}$ and then to a 3-dimensional Bessel process. The uniqueness in law can be shown by the same argument employed in the proof of Theorem \ref{t:SDEent}. The claim on the existence and the uniqueness of the SDE is basically a combination of Propositions \ref{p:step1} and \ref{p:step2}.

\begin{enumerate}
\item Note that it suffices to show
\[
a \mapsto P^{L.a}(X_{t_i} \in E_i;\, i=1, \ldots, n)
\]
is measurable for any $n$, where $0<t_1< \ldots <t_n$ are arbitrary positive real numbers and $E_i$ is an interval contained in $I$. 

Let $J:=\min\{i: t_i \geq \tau^y_{a-}\}$ and observe that $J$ depends on $a$ in a measurable way since $a \mapsto  \tau^y_{a-}$ is measurable due to the monotonicity of $\tau^y_a$ in $a$. Then,
\[
P^{L.a}(X_{t_i} \in E_i;\, i=1, \ldots, n)=E^{h,y}\left[\chf_{[X_{t_1} \in E_1]}\ldots \chf_{[X_{t_{J-1}} \in E_{J-1}]}\pi(\tau^y_{a-},J,\theta)\right],
\]
where $E^{h,y}$ is expectation with respect to $R^{h,y}$, which is the law of the solutions of (\ref{e:step1}), and 
\[
\pi(t,j,\theta)=Q^{y,\theta}(X_{(t_j-t)^+} \in E_j,X_{(t_{j+1}-t)^+} \in E_{j+1}, \ldots, X_{(t_n-t)^+} \in E_n),
\]
with $Q^{y,\theta}$ being the law of (\ref{e:sdel}) or (\ref{e:sder}) depending on the value of $\theta$. Now, since $\theta$ is independent of $B$ and $R^{h,y}$ is absolutely continuous with respect to $P^y$ with
\[
\frac{dR^{h,y}}{dP^y}=\exp\left(\frac{as'(y)}{2u(y,y)}\right)\chf_{[\tau^y_{a-}<\zeta]},
\]
as given by Proposition \ref{p:step1}, the claim follows.
\item Existence follows from constructing the solution on the  product space $C(\bbR_+,I) \times \bbR_+$. We can construct the probability space on $\Om \times \bbR_+$ by considering the  product $\sigma$-algebras generated by the measurable rectangles  and the product measure, $P^L$, obtained as the product of $P^{L,a}$ and $g$, which in particular obeys the disintegration formula. This probability space is then endowed with the natural filtration of the co-ordinate process augmented with the universal null-sets to yield a right-continuous filtration. This ensures the existence of a solution with the given properties. Uniqueness in law follows from the same time-change method of Engelbert-Schmidt applied, e.g. in Theorem \ref{t:SDEent}, previously.
\item This follows from the construction since $L^y_{\tau^y_{a-}}=a$ and $R$ does not visit $y$ after $\tau^y_{a-}$.
\end{enumerate}
\end{proof}
The following corollary establishes the connection between the solutions of (\ref{e:sdeLbridge}) and (\ref{e:sdeLrandbridge}) and the Doob-Meyer decomposition of the solutions of (\ref{e:sdeR}) when the underlying filtration is enlarged with $L^y_{\infty}$. 
\begin{corollary} \label{c:corA} Let $X$ be a regular diffusion satisfying Assumption \ref{a:X} and  $(\Om, \cG, (\cG_t)_{t \geq 0}, P)$ be a filtered probability space  in which it solves (\ref{e:sdeR}) with $X_0=y$. Consider the filtration $(\cH_t)_{t \geq 0}$, where $\cH_t=\cG_t \vee \sigma(L^y_{\infty})$. Then,
\bea
X_t&=&y+ \int_0^t\sigma(X_s)d\beta_s + \int_0^t b(X_s)ds +\int_0^{t \wedge \tau^y_{\Gamma-}}\sigma^2(X_s)\frac{u_x(X_s,y)}{u(X_s,y)}ds \nn \\
&&+\int_{t \wedge \tau^y_{\Gamma-}}^t\sigma^2(X_s)\frac{s'(X_s)}{s(X_s)-s(y)}ds, \qquad t < \zeta, \label{e:Renlarged}
\eea
where $\Gamma:=L^y_{\infty}$ and $\beta$ is an $(\cH_t)_{t \geq 0}$-Brownian motion stopped at $\zeta$.

In particular, $X$ is a weak solution of (\ref{e:sdeLrandbridge}), where $\Gamma=L^y_{\infty}$, $g(da)= \frac{s'(y)}{2u(y,y)}\exp\left(-\frac{as'(y)}{2u(y,y)}\right)da$, and $\theta=\chf_{[R_{\infty}=r]}$.
\end{corollary}
\begin{proof}
Directly following the arguments in the proof of Theorem 1.6 in \cite{MY} and keeping in mind that the $\cG_t$-conditional distribution of $L^y_{\infty}$ has an atom (as in the case of Example 1.7 in \cite{MY}), we deduce that
\[
M_t:=X_t-y-\int_0^t b(X_s)ds -\int_0^{t \wedge \tau^y_{\Gamma-}}\sigma^2(X_s)\frac{u_x(X_s,y)}{u(X_s,y)}ds -\int_{t \wedge \tau^y_{\Gamma-}}^t\sigma^2(X_s)\frac{s'(X_s)}{s(X_s)-s(y)}ds
\]
is an $(\cH_t)_{t\geq 0}$ local martingale stopped at $\zeta$. Note that the above is well-defined since, the integrand of the last integral has a constant sign and $P$-a.s.
\[
\int_0^{t }\left\{\sigma^2(X_s)+|b(X_s)|\right\}ds <\infty, \mbox{ on } [t <\zeta],
\]
by the definition of weak solutions. Therefore, on $ [t <\zeta]$, we have
\[
\int_0^{t \wedge \tau^y_{\Gamma-}}\sigma^2(X_s)\left|\frac{u_x(X_s,y)}{u(X_s,y)}\right|ds<\infty
\]
due to the continuity (and therefore boundedness) of $\frac{u_x(X,y)}{u(X,y)}$ on compact subintervals of $[0, \zeta)$.

It can also be directly verified that $[M,M]_t=\int_0^t\sigma^2(X_s)ds$ on $[t <\zeta]$. Thus, since $\sigma(x)>0$ for $x \in (l,r)$ by assumption, we may define
\[
\beta_t:=\int_0^{t \wedge \zeta}\frac{1}{\sigma(X_s)}dM_s.
\]
Observe that the above stochastic integral is well-defined since $M$ is a continuous local martingale and on $[t<\zeta]$
\[
\int_0^{t }\frac{1}{\sigma^2(X_s)}d[M,M]_s= t.
\]
By L\'evy's characterisation we easily deduce that $\beta$ is an $(\cH_t)_{t \geq 0}$-Brownian motion stopped at $\zeta$, which in turn yields that $X$ satisfies (\ref{e:Renlarged}). Moreover, $\beta$ is independent of $L^Y_{\infty}$.

To show that $X$ is the weak solution of (\ref{e:sdeLrandbridge}) with the given set of $\Gamma, g,$ and $\theta$, recall from (\ref{e:lawLT}) that the $P$-distribution of $L^y_{\infty}$ is exponential with parameter $\frac{1}{u(y,y)}$ and observe that the dynamics of $X$ in (\ref{e:Renlarged}) is the same as those given by (\ref{e:sdeLrandbridge}) once we replace $B$ with $\beta$ and notice that $[\theta=1]=[X_{t+\tau^y_{\Gamma-}} >y]$ for all $t>0$.  Moreover, $P(\theta=1)=\rho(y)$ is trivially satisfied due to the definition of $\rho$. Thus, it remains to show that $\theta$, $\beta$ and $L^y_{\infty}$ are independent from each other.

We have already observed that $\beta$ and $L^y_{\infty}$ are independent. The independence of $L^y_{\infty}$ and $X_{\infty}$ under $P$ is obvious when only one of $s(l)$ and $s(r)$ is finite. If both of them are finite, the result follows from Proposition \ref{p:LTRindep}.

Note that $(L^y_t, X_t)$ is strong Markov with respect to $(\cH_t)_{t\geq 0}$ given $L^y_{\infty}$. Thus, given $X_{\tau^y_{\Gamma-}}$ and $L^y_{\tau^y_{\Gamma-}}$, $X_{\infty}$ is independent of $(\beta_{t \wedge \tau^y_{\Gamma-}})_{t \geq 0}$. However, 
$X_{\tau^y_{\Gamma-}}=y$ and $L^y_{\tau^y_{\Gamma-}}=\Gamma=L^y_{\infty}$. Since $L^y_{\infty}$ is independent of $\beta$, we deduce that $X_{\infty}$ is independent of $(\beta_{t \wedge \tau^y_{\Gamma-}})_{t \geq 0}$, too. Moreover, $[X_{\infty}=r]=[X_t >y, \forall t >\tau^y_{\Gamma-}]$ implies $[X_{\infty}=r]\in \cH_{\tau^y_{\Gamma-}}$ since $(\cH_t)_{t \geq 0}$ is right-continuous. Therefore, $X_{\infty}$ is independent of $(\beta_{t + \tau^y_{\Gamma-}}-\beta_{\tau^y_{\Gamma-}})_{t \geq 0}$, hence, of $\beta$.

\end{proof}

Theorem \ref{t:Lbrdige} and Corollary \ref{c:corA} establish that the law of  solutions of (\ref{e:sdeLbridge}) is that of the solution of (\ref{e:sdeR}) conditioned on $[L^y_{\infty}=a]$ in view of the uniqueness in law of  solutions of (\ref{e:sdeLrandbridge}) and the disintegration formula (\ref{e:disintegration}) as stated below.
\begin{corollary} \label{c:corB}
Let $P^{L,a}$ be the law on $C(\bbR_+,I)$ induced by solutions of (\ref{e:sdeLbridge}) and consider the canonical space $(C(\bbR_+, I), \cB, P^y)$, where $\cB$ is the Borel $\sigma$-algebra on $C(\bbR_+, I)$. Then, for all $t>0$ and   any bounded and measurable $F:C([0,t],I)\mapsto \bbR$ and $h:\bbR_+\mapsto \bbR$, the following holds:
\be \label{e:RCP}
\int_0^{\infty}E^{L,a}\left[F(X_s;s\leq t)\right]h(a)\frac{1}{u(y,y)}\exp\left(-\frac{a}{u(y,y)}\right)da=E^y\left[F(X_s;s\leq t)h(L^y_{\infty})\right].
\ee
That is, $P^{L,a}$ is a regular conditional probability of $\cB$  given $L^y_{\infty}=a$.

Consequently, if $X$ is a solution of (\ref{e:sdeLrandbridge}) with $g(da)=\frac{s'(y)}{2u(y,y)}\exp\left(-\frac{as'(y)}{2u(y,y)}\right)da$, then, in its own filtration, it's a regular diffusion on $(l,r)$ with scale function $s$ and speed measure $m$. 
\end{corollary}
\begin{proof} Let $X$ be a solution of (\ref{e:sdeR}) in some filtered probability space with a right-continuous filtration. As we have seen in Corollary \ref{c:corA}, it  follows (\ref{e:Renlarged}) when the filtration is enlarged with $L^y_{\infty}$. The same corollary also yields that $X$ is a weak solution of  (\ref{e:sdeLrandbridge}) when $g(da)=\frac{s'(y)}{2u(y,y)}\exp\left(-\frac{as'(y)}{2u(y,y)}\right)da$. Since weak uniqueness holds for (\ref{e:sdeLrandbridge}), we obtain  (\ref{e:RCP}) in view of (\ref{e:disintegration}). 

The second claim follows from taking $h\equiv 1$ in (\ref{e:RCP}), in which case the left-hand side of (\ref{e:RCP}) equals
\[
E^{L,g}\left[F(R_s;s\leq t)\right].
\]
Thus, $P^{L,g}=P^y$, which implies the claim. 
\end{proof}

What is hidden, in fact, in the second part of the above corollary is a new path decomposition result for the original process $X$ which can be restated, and upgraded to a theorem, as follows:
\begin{theorem} \label{t:pathdecomp}
Let $X$ be a weak solution of (\ref{e:sdeR}) with scale function $s$ and potential kernel $u$. Pick a $y \in (l,r)$ and on a suitable probability space set up the following four independent elements:
\begin{enumerate}
\item An exponential random variable, $\Gamma$, with mean $\frac{2u(y,y)}{s'(y)}$.
\item A Bernoulli random variable, $\theta$, with $\bbP(\theta=1)=\rho(y)$, where $\rho$ is as in (\ref{e:Pexitr}).
\item A process $Y$, which is  a $(u(\cdot,y),M)$ recurrent transform of $X$ run upto $\tau^y_{\Gamma-}$, where $M_t=\exp\left(\frac{s'(y)L_t}{2u(y,y)}\right)$.
\item A pair of Bessel-type motions, $(R^0,R^1)$ with laws $(Q^{y,0},Q^{y,1})$ and lifetimes $(\zeta^0,\zeta^1)$. 
\end{enumerate}
Then, the process defined by
\[
\tilde{X}_t:=\left\{\ba{ll} 
Y_t, & t \leq \tau^y_{\Gamma-}\\
R^{\theta}_{t-\tau^y_{\Gamma-}}, & 0<t-\tau^y_{\Gamma-}\leq \zeta^{\theta},
\ea
\right.
\]
has the same law as $X$.
\end{theorem}
The above path decomposition result yields in particular an algorithm to simulate transient diffusions.  Note that the random variable $\theta$ is needed only if $s(l)=1-s(r)=0$. It must be emphasised that, in this case,  the Bessel-type motions in the above setup are not independent from each other. In fact, the underlying Brownian motions in their SDE representations are the same. This is indeed possible since (\ref{e:sdel}) and (\ref{e:sder}) have the same coefficients and  differ only in their choice of the state space.

\section{Examples} \label{s:examples}
In this section we present some explicit examples that follow from Theorems \ref{t:Lbrdige} and \ref{t:pathdecomp}.
\subsection{Killed Brownian motion} Suppose $X$ is a Brownian motion on $(-\infty,b)$ killed at $b>0$. Thus, $s(x)=\frac{x}{b}$,  $u(x,y)= 1- s(x \vee y)$, and $X_{\infty}=b$, which is reached in finite time, a.s.. Taking  $y=0$ the equation (\ref{e:sdeLbridge}) reads as
\be \label{e:condkBM}
X_t=B_t - \int_0^{t\wedge \tau_{a-}^0} \frac{1}{b-X_s}\chf_{[X_s > 0]}ds +\int_{t\wedge \tau_{a-}^0}^t \frac{1}{X_s}ds, \qquad t < \zeta,
\ee
where $\zeta$ is the first hitting time of $b$, which occurs in finite time. 

The first integral represents the recurrent transform, $Y$, stopped at $\tau^0_{a-}$, where
\[
Y_t= B_t - \int_0^{t} \frac{1}{b-Y_s}\chf_{[Y_s >0]}ds.
\]
If we let $U:=b-Y$, then
\be \label{e:recKBU}
U_t= b+ \beta_t +\int_0^t \frac{1}{U_s}\chf_{[U_s<b]}ds,
\ee
where $\beta=-B$. Although $U$ resembles a $3$-dimensional process, it clearly isn't since it is recurrent. It behaves like a Bessel process on the interval $(0,b)$, otherwise it moves like a Brownian motion, which makes it recurrent. This hints at the guess $U$ being a recurrent transform of a Bessel process, which turns out to be true.
\begin{proposition} Let $U$ be as in (\ref{e:recKBU}). Then, it is an $(h,M)$-recurrent transform of the $3$-dimensional Bessel process, $R$, which solves 
\be \label{e:3dB}
R_t=R_0+B_t +\int_0^t \frac{1}{R_s}ds,
\ee
where 
\[
h(x)= \frac{1}{x \vee b}, \; M_t= \exp\left(\frac{L^b_t}{2b}\right),
\]
and $L^b$ is the semimartingale local time of $R$ at $b$.
\end{proposition} 
\begin{proof} The potential kernel, $u$, for $R$ is given by
\[
u(x,y)=\frac{1}{x\vee y}
\]
while its scale function is $1-\frac{1}{x}$. 

Proposition \ref{p:step1} yields that  $(u(\cdot,b), M)$ is a recurrent transform which results in the SDE
\[
dX_t=dB_t + \left\{\frac{1}{X_t}-\chf_{[X_t>b]}\frac{1}{X_t}\right\}dt=dB_t + \chf_{[X_t\leq b]}\frac{1}{X_t}dt.
\]
Since $[X_t=b]$ is a null set, $\int_0^t \chf_{[X_s=b]}ds=0$,  which in turn yields the claim. 
\end{proof}
Having characterised the recurrent transform and the Bessel process that led to (\ref{e:condkBM}), we can now state the path decomposition for the killed Brownian motion using Theorem \ref{t:pathdecomp}.
\begin{corollary}  On a suitable probability space set up the following three independent elements:
\begin{enumerate}
\item An exponential random variable, $\Gamma$, with mean $2b$.
\item A weak solution, $U$, of (\ref{e:recKBU}).
\item A $3$-dimensional Bessel process, $R$, which solves (\ref{e:3dB}) with $R_0=0$. 
\end{enumerate}
Consider
\[
\tilde{X}_t:=\left\{\ba{ll} 
b-U_t, & t \leq \tau^b_{\Gamma-}\\
R_{t-\tau^b_{\Gamma-}}, & 0<t-\tau^y_{\Gamma-}\leq S_b, 
\ea
\right.
\]
where $(\tau^b_t)_{t \geq 0}$ is the right-continuous inverse of the local time of $U$ at level $b$ and $S_b:=\inf\{t\geq 0: R_t=b\}$. Then, $\tilde{X}$ has the same law as the Brownian motion starting at $0$ and killed at $b$. 
\end{corollary}

\subsection{Ornstein-Uhlenbeck process}
Consider the Ornstein-Uhlenbeck process
\be 
X_t=y+ B_t+ \int_0^t (r X_s +b)ds
\ee
for some $r>0$. Then, $X$ is transient and has the scale function 
\be\label{e:OUscale}
s(x)=\sqrt{\frac{r}{\pi}}\int_{-\infty}^x \exp\left(-r\left(y+\frac{b}{r}\right)^2\right)dy,
\ee
with $s(-\infty)=1-s(\infty)=0$. The equation (\ref{e:sdeLbridge}) reads in this case as
\bea 
X_t &=& y +B_t  + \int_0^t (r X_s +b)ds \nn  \\
&&+ \int_0^{t \wedge \tau^y_{a-}} \left\{ \frac{s'(X_s)}{s(X_s)}\chf_{[X_s \leq y]} -\frac{s'(X_s)}{1-s(X_s)}\chf_{[ X_s > y]} \right\}ds \label{e:OULyBridge} \\
&&+ \int_{t \wedge \tau^y_{a-}}^t\sigma^2(X_s)\left\{\theta \chf_{[X_s> y]}\frac{s'(X_s)}{s(X_s)-s(y)}-(1-\theta)\chf_{[X_s< y]}\frac{s'(X_s)}{s(y)-s(X_s)}\right\}ds. \nn
\eea

We already know that there exists a  weak solution, which is unique in law. The SDE above in  fact  possesses a unique strong solution. Indeed, since $\sigma \equiv 1$,  Lemma IX.3.3, Corollary IX.3.4 and Proposition IX.3.2 in \cite{RY} imply that pathwise uniqueness holds for the SDEs  (\ref{e:sdel}) and (\ref{e:sder}) associated with the Ornstein-Uhlenbeck process above. Moreover, the auxiliary SDE
\[
dX_t= d\beta_t + \left\{ \frac{s'(X_t)}{s(X_t)}\chf_{[ X_t \leq y]} -\frac{s'(X_t)}{1-s(X_t)}\chf_{[ X_t > y]}\right\}dt
\]
has pathwise uniqueness until the first exit time from any bounded interval  by part i of Theorem IX.3.5 since the drift coefficient is bounded in compact subsets of $\bbR$.   This establishes the  pathwise uniqueness for the solutions of (\ref{e:step1}). Thus, in view of the celebrated result of Yamada and Watanabe (see Corollary 5.3.23 in \cite{KS}), there exists a unique strong solution to (\ref{e:step1}), hence, to (\ref{e:OULyBridge}).

\subsection{Squared Bessel process} Now, $X$ is a squared Bessel process on $(0, \infty)$ of order $\delta>2$, i.e.
\[
X_t=y+ \int_0^t 2\sqrt{X_s}dB_s + \delta t.
\]
Note that a scale function is given by $s(x)= 1-x^{\frac{2-\delta}{2}}$. Thus, the equation (\ref{e:sdeLbridge}) reads
\bean
X_t&=&y+ \int_0^t 2\sqrt{X_s}dB_s + \delta t \\
&&-2(\delta-2) \int_0^{t \wedge \tau^y_{a-}}\chf_{[X_s>y]}ds  + 2(\delta-2) \int_{t \wedge \tau^y_{a-}}^t\frac{X_s^{\frac{2-\delta}{2}}}{y^{\frac{2-\delta}{2}}-X_s^{\frac{2-\delta}{2}}}ds.
\eean
Observe that we do not need to introduce the random variable $\theta$ since $\rho(y)=1$.

As in the previous example we can show that the solution is in fact the unique strong solution once we show that the following SDE has pathwise uniqueness:
\[
X_t=y+ \int_0^t 2\sqrt{X_s}dB_s + \delta t -2(\delta-2) \int_0^{t}\chf_{[X_s>y]}ds.
\]
Indeed, if $R$ and $X$ are two strong solutions then
\bean
|R_t -X_t| &=& 2 \int_0^t \mbox{sgn}(R_s-X_s) (\sqrt{R_s}-\sqrt{X_s})dB_s\\
&& -2(\delta-2)\int_0^{t} \mbox{sgn}(R_s-X_s)  (\chf_{[R_s>y]}-\chf_{[X_s>y]})ds + L_t\\
&\leq & \int_0^t \mbox{sgn}(R_s-X_s) (\sqrt{R_s}-\sqrt{X_s})dB_s  + L_t,
\eean
where $L$ is the semimartingale local time of $R-X$ at $0$ and $\mbox{sgn}(x)= 1$ if $x > 0$ and $-1$, otherwise. If $\tau_x:=\inf\{t\geq 0: |R_t-X_s|>x\}$, then the stochastic integral stopped at $\tau_x$ is a true martingale since $|\sqrt{z}-\sqrt{z'}|< \sqrt{|z-z'|}|$. Thus,
\[
E|R_{t\wedge\tau_x}-X_{t\wedge\tau_x}| \leq E L_{t\wedge \tau_x}.
\]
However, $L\equiv 0$ by Lemma IX.3.3 in \cite{RY} since, again, $|\sqrt{z}-\sqrt{z'}|< \sqrt{|z-z'|}|$. 
\bibliographystyle{siam}
\bibliography{ref}
\appendix
\section{Proof of Theorem \ref{t:SDEent}}
\begin{proof}[Proof of Theorem \ref{t:SDEent}]
Suppose that $l$ is an entrance boundary and  assume $l=0$ without loss of generality.  Since $0$ is entrance, we must have $s(0)=-\infty$. Although $s(r)=1$ by assumption, let's apply a further affine transformation and assume $s(r)=0$. Also note that if $r$ is a singular boundary,  transience of  $X$ implies $r$ is not entrance.
Consider a $3$-dimensional Bessel process, $R$, on $(0,\infty)$  starting at $0$, i.e
\[
R_t= B_t +\int_0^t \frac{1}{R_s}ds.
\]
It is well-known that $0$ is an entrance boundary for $R$ (see \cite{BorSal} for  a summary of results on Bessel processes), whose scale function is given by $p(x):=-\frac{1}{x}$. We will construct a weak solution of (\ref{e:SDEent}) by change of time and scale applied to $R$ following the ideas of Engelbert and Schmidt (see, e.g., Section 5.5 of \cite{KS}). 

Observe that $s$ is a $C^1$-function whose derivative is absolutely continuous. It is clear that these properties are inherited by its inverse, $s^{-1}$. Since $s^{-1}(-\infty)=0$, it follows from Ito's formula and the fact that 
\[
\frac{\sigma^2}{2}s'' + b s'=0,
\]
 $Y=s^{-1}(p(R))$ is a semimartingale satisfying
\be \label{e:Besscaled}
Y_t=\int_0^t \frac{1}{R_u^2 s'\left(s^{-1}\left(p(R_u)\right)\right)}dB_u + \int_0^t \frac{b\left(s^{-1}\left(p(R_u)\right)\right)}{\sigma^2\left(s^{-1}\left(p(R_u)\right)\right)\left[R_u^2 s'\left(s^{-1}\left(p(R_u)\right)\right)\right]^2}du
\ee
Consider the additive functional
\[
A_t:=\int_0^t \frac{1}{\sigma^2\left(s^{-1}\left(p(R_u)\right)\right)\left[R_u^2 s'\left(s^{-1}\left(p(R_u)\right)\right)\right]^2}du,
\]
and its right-continuous inverse
\[
T_t:=\inf\{s\geq 0: A_s >t\}.
\]
We also define $A_{\infty}=\lim_{t \rar \infty} A_t$ and $T_{\infty}=\lim_{t \rar \infty} T_t$.

It follows from the occupation times formula that
\[
A_t=2 \int_0^{\infty} \frac{\tilde{l}^x_t}{x^2\sigma^2\left(s^{-1}\left(p(x)\right)\right)\left(s'(s^{-1}(p(x))\right)^2}dx=2\int_0^r \frac{\tilde{l}^{-\frac{1}{s(x)}}_t}{\sigma^2\left(x\right)s'(x)}dx,
\]
where $\tilde{l}^x$ is the diffusion local time for $R$. Since $R_t \rar \infty$, a.s. and never visits $0$ again, we immediately deduce that for every $t$, the  mapping $x \mapsto \tilde{l}^x_t$ has a compact support, a.s., that is contained in $(0,\infty)$. Since it is also \cadlag, it is bounded. This implies that $A_t<\infty$, a.s., since for any $0<z<u<r$
\[
2\int_z^u \frac{1}{\sigma^2\left(x\right)s'(x)}dx = m((z,u))<\infty.
\]
The same reasoning also yields that $A_{T_x}<\infty$, where $T_x= \inf\{t\geq 0:R_t=x\}$ for some $x\in (0,\infty)$. Thus, the strong Markov property implies
\[
Q^0(A_{\infty}<\infty)=Q^x(A_{\infty}<\infty),
\]
where $Q^x$ is the law of a $3$-dimensional Bessel process starting at $x$.  Mijatovic and Urusov show in Theorem 2.11  \cite{MU} that  $Q^x(A_{\infty}<\infty)=1$ (resp. $=0$) if 
\[
\int_z^{\infty} \frac{1}{x^3\sigma^2\left(s^{-1}\left(p(x)\right)\right)\left(s'(s^{-1}(p(x))\right)^2}dx <\infty \mbox{ (resp. $=\infty$)})
\]
for some $z$.
After a change of variable the above integral turns into
\[
-\int_{s^{-1}(p(z))}^r\frac{s(x)}{\sigma^2(x)s'(x)}dx,
\]
which is finite if $r$ is a regular boundary. Consequently, $A_{\infty}<\infty$, a.s. when $r$ is regular.

If $r$ is a singular boundary, we have already observed at the beginning of the proof that it is not entrance. If it is exit, the  above integral will be finite, too. Indeed, that $r$ is exit implies
\[
\int_z^r m((z,x))s'(x)dx <\infty.
\]
Also, note that $\lim_{x \rar r} s(x)m((z,x)$ exists and equals
\[
-\int_{z}^r\frac{s(x)}{\sigma^2(x)s'(x)}dx-\int_z^r m((z,x))s'(x)dx
\]
under the assumption that $r$ is an exit boundary. On the other hand, since $s(r)=0$ and 
\[
\int_z^r \frac{s'(x)}{s(x)}dx=\log s(z)-\log s(r)=\infty,
\]
we deduce that $\lim_{x \rar r} s(x)m((z,x)=0$ since 
\[
\infty > \int_z^r m((z,x))s'(x)dx=\int_z^r s(x)m((z,x))\frac{s'(x)}{s(x)}dx.
\]
Therefore, for any $z \in (0,r)$
\[
-\int_{z}^r\frac{s(x)}{\sigma^2(x)s'(x)}dx=\int_z^r m((z,x))s'(x)dx<\infty,
\]
and $A_{\infty}$ is finite. 

If $r$ is a natural boundary, let $r_n$ be a sequence of numbers in $(l,r)$ increasing to $r$ and observe that
\bean
-\int_{z}^r\frac{s(x)}{\sigma^2(x)s'(x)}dx&=& \lim_{n \rar \infty} \left\{-s(r_n)m((z,r_n)) + \int_z^{r_n} m((z,x))s'(x)dx\right\}\\
&\geq & \lim_{n \rar \infty} \int_z^{r_n} m((z,x))s'(x)dx=\infty.
\eean
Thus, $A_{\infty}=\infty$, a.s., when $r$ is a natural boundary. 

The behaviour of $A$ near infinity affects the finiteness of $T$. If $A_{\infty}<\infty$, then $T_t =\infty$ on $[t \geq A_{\infty}]$. Otherwise, $T_t <\infty$, for every $t$. Moreover, it is easy to see that $A$ is continuous and strictly increasing while $T$ is  continuous everywhere and strictly increasing on $[0, A_{\infty})$. Moreover, $t=A_{T_t}$ for $t \leq  A_{\infty}$ and $t=T_{A_{t}}$ for every $t <\infty$.

With the above characterisation of  $A$ and $T$, let us next consider $X_t:= Y_{T_t}$ and $\cG_t=\cF_{T_t}$, where $(\cF_t)$ is the universal completion of the natural filtration of $R$. Consequently, $(\cG_t)$ satisfies the usual conditions since $(T_t)$ is continuous. Then, it is straightforward to check that 
\[
M_t:=\int_0^{T_t} \frac{1}{R_u^2 s'\left(s^{-1}\left(p(R_u)\right)\right)}dB_u
\]
is a continuous $(\cG_t)$-local martingale with
\[
[M,M]_t=\int_0^{T_t} \frac{1}{\left[R_u^2 s'\left(s^{-1}\left(p(R_u)\right)\right)\right]^2}du.
\]
On the set $[T_t<\infty]=[t<A_{\infty}]$, the above can be rewritten as
\[
[M,M]_t=\int_0^{T_t} \sigma^2(s^{-1}(p(R_u)))dA_u=\int_0^t \sigma^2(X_u)du.
\]
Thus, there exists a $(\cG_t)$-Brownian motion, $\beta$, such that
\[
M_t =\int_0^t \sigma(X_u)d\beta_u.
\]
Similarly, on $[t<A_{\infty}]$
\[
\int_0^{T_t} \frac{b\left(s^{-1}\left(p(R_u)\right)\right)}{\sigma^2\left(s^{-1}\left(p(R_u)\right)\right)\left[R_u^2 s'\left(s^{-1}\left(p(R_u)\right)\right)\right]^2}du=\int_0^t b(X_u)du.
\]
Thus, we have proved on $[t<A_{\infty}]$
\[
X_t=\int_0^t \sigma(X_u)d\beta_u+ \int_0^t b(X_u)du.
\]
On the other hand, on $[t \geq A_{\infty}]$, $T_t=\infty$, and $X_t= Y_{\infty}=s^{-1}(p(R_{\infty}))=s^{-1}(p(\infty))=r$. Furthermore, since $s^{-1}\circ p$ is one-to-one, $A_{\infty}=\inf\{t \geq 0: X_t=r\}$. This shows the existence of a weak solution to (\ref{e:SDEent}) as soon as we verify that $\zeta =A_{\infty}$. However, this immediately follows from the fact that any diffusion that satisfies (\ref{e:SDEent}) has the scale and speed given by $s$ and $m$, respectively, which in turn yields that $0$ is an entrance boundary  and, therefore, inaccessible. 

To show uniqueness let $X$ be a weak solution and $Y=p^{-1}(s(X))$ so that on $[t <\zeta]$
\[
Y_t= \int_0^t \frac{s'(X_u)\sigma(X_u)}{s^{2}(X_u)}dB_u -\int_0^t \frac{\left(s'(X_u)\sigma(X_u)\right)^2}{s^3(X_u)}du.
\]
Consider
\[
T_t:=\int_0^t \frac{\left(s'(X_u)\sigma(X_u)\right)^2}{s^4(X_u)}du
\]
and its right continuous inverse $A_t:=\inf\{s \geq 0: T_s>t\}$. 

As before,
\[
T_t=2 \int_0^r\tilde{L}^x_t \frac{s'(x)}{s^4(x)}dx,
\]
where $\tilde{L}^x$ is the diffusion local time at $x$ for $X$. On the set $[t<\zeta]$, $\tilde{L}^x$ has compact support in $(0,r)$. Thus, $T_t<\infty$ on $[t<\zeta]$ since for $0<z<u<r$
\[
3\int_z^u \frac{s'(x)}{s^4(x)}dx= \frac{1}{s^3(z)}-\frac{1}{s^3(u)}<\infty.
\]
Similarly, $T_t$ is absolutely continuous and strictly increasing on $[0,\zeta)$.

Next observe that
\[
-\int_z^r  \frac{s(x)}{s'(x)\sigma^2(x)}\frac{\left(s'(x)\sigma(x)\right)^2}{s^2(x)}dx=-\int_z^r\frac{s'(x)}{s(x)}dx=\infty.
\]
Thus, Theorem 2.11 in \cite{MU} yields $T_t=\infty$ on $[t \geq \zeta]$. This, in particular, yields that $A_t<\infty$ for all $t<\infty$ and $A_{\infty}=\zeta$. 

Consider $X_t=Y_{A_t}$ and $\cG_t=\cF_{A_t}$, where $(\cF_t)$ is the universal completion of the natural filtration of $X$. Define the $(\cG_t)$-local martingale
\[
M_t= \int_0^{A_t} \frac{s'(X_u)\sigma(X_u)}{s^{2}(X_u)}dB_u 
\]
with
\[
[M,M]_t=\int_0^{A_t}  \frac{(s'(X_u))^2\sigma^2(X_u)}{s^{4}(X_u)}du=T_{A_t}.
\]
Thus, on the set $[A_t <\zeta]$, $[M,M]_t=t$ as well as
\[
-\int_0^{A_t} \frac{\left(s'(X_u)\sigma(X_u)\right)^2}{s^3(X_u)}du= \int_0^t \frac{1}{X_u}du.
\]
Let $\tau:=\inf\{t\geq 0: A_t=\zeta\}$. The above considerations show that on $[0,\tau)$
\[
X_t=\beta_t+\int_0^t \frac{1}{X_u}du.
\]
Using the continuity of $X$ we also deduce that $X_{\tau}=Y_{\zeta}=r$ and $X$ is, therefore, $3$-dimensional Bessle process starting at $0$ and stopped at $r$. Since the SDE for the $3$-dimensional Bessel process has a unique solution, we deduce that the distribution of $(Y_{A_t})$ is uniquely identified on $[A_t<\zeta]$. This in turn yields the weak uniqueness of the solutions of (\ref{e:SDEent}) on $[0,\zeta]$ since $A_t$ is strictly increasing on $[t\leq \tau]$ and $p^{-1}\circ s$ is one-to-one.  

If the entrance boundary is the right endpoint, i.e. $r$, suppose without loss of generality $r=0$ and consider the diffusion, $R$,  on $(-\infty, 0)$ defined by the SDE:
\[
R_t= B_t +\int_0^t\frac{1}{R_s}ds.
\]
This is the negative of a $3$-dimensional Bessel process and $0$ is its entrance boundary. Now, the above arguments can be repeated  to show the existence and the uniqueness of the weak solutions of (\ref{e:SDEent}).
\end{proof}

\end{document}